\documentclass[11pt,a4paper]{article}

\usepackage[utf8]{inputenc}
\usepackage[T1]{fontenc}
\usepackage{lmodern}
\usepackage[margin=1in]{geometry}
\usepackage{amsmath,amssymb,amsthm}
\usepackage{mathtools}
\usepackage{mathrsfs}
\usepackage{bm}
\usepackage{graphicx}

\usepackage[hidelinks]{hyperref}
\hypersetup{
  pdftitle={A covering construction of labelled packing measure and content},
  pdfauthor={Peizhi Liu},
  pdfsubject={Labelled packing measure and content}
}

\theoremstyle{plain}
\newtheorem{theorem}{Theorem}[section]
\newtheorem{lemma}[theorem]{Lemma}
\newtheorem{proposition}[theorem]{Proposition}
\newtheorem{corollary}[theorem]{Corollary}

\theoremstyle{definition}

\theoremstyle{remark}
\newtheorem{remark}[theorem]{Remark}

\numberwithin{equation}{section}

\newcommand{\N}{\mathbb{N}}

\newcommand{\R}{\mathbb{R}}

\newcommand{\K}{\mathcal K}
\newcommand{\Kz}{\mathcal K_0}
\newcommand{\Prob}{\operatorname{Pr}}

\newcommand{\proj}{\operatorname{proj}}

\newcommand{\MM}{\mathcal M}
\newcommand{\dimH}{\dim_{\mathrm H}}
\newcommand{\dimP}{\dim_{\mathrm P}}

\DeclareMathOperator{\diam}{diam}

\DeclareMathOperator{\dist}{dist}

\title{A covering construction of labelled packing measure and content}
\author{Peizhi Liu\\[0.5em]
  \small School of Mathematics and Statistics,\\
  \small Nanjing University of Science and Technology,\\
  \small Nanjing 210094, China\\
  \small \href{mailto:liupeizhi@njust.edu.cn}{\texttt{liupeizhi@njust.edu.cn}}}
\date{}

\begin{document}
\maketitle
\label{firstpage}

\begin{abstract}
We introduce labelled packing measure by a covering construction in which each set has an independent positive label bounding its diameter and the radii in its local packing cost. Letting the labels tend to zero defines a metric outer measure $\mathcal M^s$. For every $s>0$ and every metric space $(X,d)$, we prove
\[
 2^{-s}\mathcal P^s(F)\le\mathcal M^s(F)
 \le C(s)\mathcal P^s(F)\qquad(F\subset X),
\]
where $\mathcal P^s(F)$ is the classical packing measure and $C(s)$ depends only on $s$. Thus labelled packing measure and classical packing measure have the same null sets and critical exponent. Allowing arbitrary finite positive labels defines labelled packing content. At the common positive Hausdorff and packing dimension, this content and the labelled packing measure agree on cylinders associated with irreducible subshifts of finite type. In particular, equality holds for non-singleton self-similar attractors and components of finite strongly connected graph-directed systems, without separation assumptions.
\end{abstract}


\section{Definitions and Main Results}

\subsection{Labelled packing measure}

Hausdorff measure is defined by coverings of small diameter. Packing
measure is obtained in two stages: a supremum over disjoint balls defines
the packing premeasure, which is then regularized by countable coverings.
This construction was introduced by Tricot~\cite{Tricot1982} and
developed by Sullivan~\cite{Sullivan1984} and Taylor and
Tricot~\cite{TaylorTricot1985}; see also \cite{Falconer2014}.
In the terminology of Rogers~\cite{Rogers1970}, let \(\tau\) be
a non-negative set function with \(\tau(\varnothing)=0\). Its Method I
construction is the covering envelope
\[
 \tau^{\mathrm I}(F)
 :=\inf\left\{\sum_{j=1}^{\infty}\tau(E_j):
 F\subseteq\bigcup_{j=1}^{\infty}E_j\right\}.
\]
For Method II, one first restricts the diameters of the covering sets. Thus
\[
 \tau^{\mathrm{II}}_\delta(F)
 :=\inf\left\{\sum_{j=1}^{\infty}\tau(E_j):
 F\subseteq\bigcup_{j=1}^{\infty}E_j,\ |E_j|\leq\delta\right\},
 \qquad
 \tau^{\mathrm{II}}(F):=\lim_{\delta\downarrow0}
 \tau^{\mathrm{II}}_\delta(F).
\]
Method I imposes countable subadditivity in a single covering step. Method II
also localizes the construction by forcing the covering diameters to tend to
zero. Hausdorff measure is the standard example of Method II. Classical
packing measure uses Method I after the packing premeasure has already been
formed.

Let \((X,d)\) be a metric space, let \(\widehat X\) be its completion,
and let all balls be closed balls in \(X\). For \(E\subset X\), write
\[
        |E|:=\diam E,
\]
with \(|\varnothing|=0\), and allow \(|E|=\infty\).

We first recall the definition of classical packing measure.
A packing is a finite or countable pairwise disjoint family of balls.
For \(s>0\) and
\(\delta\in(0,\infty]\), define the \(\delta\)-packing precontent by
\[
\mathcal P^{s}_{\delta}(E):=\sup\left\{\sum_i (2r_i)^{s}:
 \{B(x_i,r_i)\}_i\text{ is a packing},\ x_i\in E,\
 0<r_i<\infty,\ r_i\leq\delta\right\}.
\]
The packing
premeasure is
\(\mathcal P^{s}_{0}(E):=\inf_{\delta>0}\mathcal P^{s}_{\delta}(E)\).
The \emph{\(s\)-dimensional packing measure} \(\mathcal P^s\) is the
Method I regularization of \(\mathcal P^s_0\):
\[
\mathcal P^s(F)
:=
\inf\left\{
        \sum_{j=1}^\infty \mathcal P^s_0(F_j):
        F\subseteq \bigcup_{j=1}^\infty F_j
\right\}.
\]

We now define labelled packing measure by a Method II construction
in which each covering set is assigned a positive label.
The label bounds both the diameter of the set and the radii used to
define its local packing cost.

For \(s>0\), nonempty \(E\subset X\), and \(0<R<\infty\) with
\(|E|\le R\), define the \emph{local packing cost} by
\begin{equation}
\label{eq:label-cost}
 \Phi^s(E,R):=
 \sup\left\{\sum_{i=1}^N(2r_i)^s:
 \begin{array}{l}
 x_i\in\overline E^{\widehat X},\quad 0<r_i\le R,\\
 d(x_i,x_j)\ge r_i+r_j\quad(i\ne j)
 \end{array}\right\}.
\end{equation}
Set
\(\Phi^s(\varnothing,R)=0\). For \(0<\delta<\infty\), define
\begin{equation}
\label{eq:label-content}
 \mathcal M^s_\delta(F):=
 \inf\left\{\sum_{j=1}^{\infty}\Phi^s(E_j,R_j):
 F\subseteq\bigcup_{j=1}^{\infty}E_j,\quad
 |E_j|\le R_j\le\delta,\quad R_j>0\right\}.
\end{equation}
The \emph{labelled packing measure} is defined by
\begin{equation}
\label{eq:label-measure}
 \mathcal M^s(F):=
 \lim_{\delta\downarrow0}\mathcal M^s_\delta(F)
 =\sup_{0<\delta<1}\mathcal M^s_\delta(F).
\end{equation}

Our first main result shows that this Method II construction recovers packing
measure up to a constant depending only on the exponent. 

\begin{theorem}
\label{thm:comparison}
Let \(s>0\), and let \((X,d)\) be a metric space. There exists a constant
\(C(s)\in(0,\infty)\), depending only on \(s\), such that
\[
        2^{-s}\mathcal P^{s}(F)
        \leq
        \mathcal M^{s}(F)
        \leq
        C(s)\mathcal P^{s}(F)
        \qquad(F\subset X).
\]
In particular, setting
\[
        \dim_{\mathcal M}F
        :=\inf\{t>0:\mathcal M^{t}(F)=0\},
\]
one has \(\dim_{\mathcal M}F=\dim_P F\).
\end{theorem}

The upper estimate follows from a finite-scale bound obtained by
iteration. At each stage, the stopping-time construction covers part of
the set with controlled total cost and reduces the packing precontent
of the remainder by a fixed factor. Summing the resulting estimates and
passing to the scale limit gives the comparison.

\subsection{Labelled packing content}

Hausdorff content $\mathcal H^s_\infty$ is obtained from the covering
formula for $\mathcal H^s$ by omitting the restriction on covering
diameters. Farkas and Fraser~\cite[Theorem~2.1 and
Corollaries~2.3--2.4]{FF} proved equality between Hausdorff content
and Hausdorff measure at the Hausdorff dimension for cylinders
associated with irreducible subshifts of finite type and for the
sets $G_i$ defined below. In particular, their results apply to
self-similar sets and components of strongly connected graph-directed
systems of similarities, without separation assumptions.

They also discussed possible definitions of packing content
in~\cite[Section~5]{FF}. Simply removing the radius bound from the
packing supremum gives infinity on every nonempty set when $s>0$.
They suggested requiring at least two balls or bounding the radii by
the diameter of the set, but did not develop these suggestions into
a covering content. Their packing results give equality between
$\mathcal P^s_\delta$ and $\mathcal P^s$ at the packing dimension,
for all sufficiently small $\delta>0$ under strong separation.

The independent labels in~\eqref{eq:label-content} provide a covering
formulation: each local packing cost has a finite radius bound, while
the covering construction imposes countable subadditivity.
For $F\subset X$, the \emph{labelled packing content} is defined by
allowing arbitrary finite positive labels:
\begin{equation}\label{content:full}
 \MM^s_\infty(F):=
 \inf\left\{\sum_j\Phi^s(E_j,R_j):
 F\subset\bigcup_jE_j,\quad
 |E_j|\le R_j,\quad 0<R_j<\infty\right\}.
\end{equation}

\begin{corollary}\label{content:basic}
For $s>0$, $A\subset\R^d$, and $0<\delta<\infty$,
\begin{equation}\label{content:chain}
 \mathcal P^s(A)\le\MM^s_\infty(A)
 \le\MM^s_\delta(A)\le\MM^s(A)
 \le C(s)\mathcal P^s(A).
\end{equation}
\end{corollary}

In particular, $\MM^s_\infty$, $\MM^s$, and $\mathcal P^s$ have the
same null sets, and
\begin{equation}\label{content:dimension}
 \inf\{s>0:\MM^s_\infty(A)=0\}=\dimP A.
\end{equation}

Our second main result gives equality between labelled packing
content and labelled packing measure. Let $(S_i)_{i\in I}$ be a finite
family of contracting similarities of $\R^d$, and let $\Sigma_A$ be
the subshift of finite type determined by a zero--one matrix
$A=(A_{ij})_{i,j\in I}$. With $\Pi$ denoting the coding map, write
\[
 F_A=\Pi(\Sigma_A),\qquad
 F_A^w=\Pi(\Sigma_A\cap[w]),\qquad
 G_i=\bigcup_{j:A_{ij}=1}F_A^j,
\]
where $[w]$ is the cylinder corresponding to a finite nonempty word
$w$. The precise definitions of $\Sigma_A$, $\Pi$, $[w]$, and
irreducibility of $A$ are given in Section~\ref{content:section}.

\begin{theorem}\label{content:subshift}
Let $A$ be irreducible and let $s=\dimH F_A>0$. Then
$\dimP F_A=s$. For $K=F_A^w$ or $K=G_i$ and every
$0<\delta<\infty$,
\begin{equation}\label{content:subshift-equality}
 \MM^s_\infty(K)=\MM^s_\delta(K)=\MM^s(K)<\infty.
\end{equation}
\end{theorem}

\begin{corollary}\label{content:strongM}
Let $E\subset\R^d$ be a non-singleton self-similar attractor. Then
$s:=\dimH E=\dimP E>0$, and for every finite $\delta>0$,
\begin{equation}\label{content:strongEquality}
 \MM^s_\infty(E)=\MM^s_\delta(E)=\MM^s(E)<\infty.
\end{equation}
No separation condition is assumed.
\end{corollary}

\begin{corollary}\label{content:graph}
Let $(E_i)_{i\in V}$ be the non-singleton compact components of a
finite strongly connected graph-directed system of contracting
similarities,
\[
 E_i=\bigcup_j\bigcup_{e\in\mathscr E_{ij}}S_e(E_j).
\]
Their common Hausdorff and packing dimension is $s>0$. For every
$i\in V$ and every finite $\delta>0$,
\[
 \MM^s_\infty(E_i)=\MM^s_\delta(E_i)=\MM^s(E_i)<\infty.
\]
No separation condition is assumed.
\end{corollary}

The proof uses the fact that every sufficiently small ball centred
on $G_i$ contains a similarity image of $G_i$ whose diameter is
comparable to the radius of the ball. Irreducibility gives this
property, which implies finiteness at the common Hausdorff and
packing dimension. When $\MM^s(G_i)>0$, a packing-density estimate
provides pairwise disjoint similarity images of arbitrarily small
diameter covering $G_i$ up to a set of zero $\MM^s$-measure; their
ratios satisfy $\sum_n b_n^s=1$. By similarity covariance of $\Phi^s$,
applying these similarities to a labelled cover with bounded labels
preserves its total cost and bounds all labels by any prescribed
positive scale. This proves equality of content
and measure on $G_i$, and similarity covariance gives the cylinder
assertion. The case $\MM^s(G_i)=0$ follows from~\eqref{content:chain}.

\section{Fundamental Properties}
\label{sec:fundamental-properties-M}

This section establishes the basic properties of \(\mathcal M^s_\delta\),
the labelled packing content \(\mathcal M^s_\infty\), and the labelled
packing measure \(\mathcal M^s\), including their measurability on
compact hyperspaces.

\subsection{Basic properties}

\begin{lemma}
\label{lem:capacity-packing}
For \(s>0\), \(0<R<\infty\), and \(E\subset X\) with \(|E|\le R\), 
\begin{equation}
\label{eq:capacity-packing}
 \Phi^s(E,R)\le\mathcal P_R^s(E)\le2^s\Phi^s(E,R).
\end{equation}
If \(X=\mathbb R^d\), then
\begin{equation}\label{content:exactPhi}
 \Phi^s(E,R)=\mathcal P^s_R(E).
\end{equation}
\end{lemma}

\begin{proof}
The case \(E=\varnothing\) is immediate. Fix a finite admissible
configuration \((x_i,r_i)_{i=1}^N\) in
\(\overline E^{\widehat X}\) and \(t\in(0,1)\). Choose \(y_i\in E\)
with \(d(x_i,y_i)<(1-t)\min_i r_i/4\). Then
\[
 d(y_i,y_j)
 \ge d(x_i,x_j)-d(x_i,y_i)-d(x_j,y_j)
 >t(r_i+r_j)\qquad(i\ne j).
\]
Thus \((y_i,tr_i)_i\) is admissible with centres in \(E\), and the
closed balls \(B(y_i,tr_i)\) are pairwise disjoint. Letting
\(t\uparrow1\) and taking the supremum proves  \(\Phi^s(E,R)\le\mathcal P_R^s(E)\).

Conversely, pairwise disjoint closed balls satisfy
\[
 d(x_i,x_j)>\max\{r_i,r_j\}\ge(r_i+r_j)/2.
\]
Consequently, replacing each radius \(r\) in an \(R\)-packing of \(E\) 
by \(r/2\) yields an admissible separated configuration.
Taking finite subsums and then
suprema yields \(\mathcal P_R^s(E)\le2^s\Phi^s(E,R)\).
In Euclidean space, disjoint closed balls satisfy
\(|x_i-x_j|>r_i+r_j\), so finite subsums give
\(\mathcal P_R^s(E)\le\Phi^s(E,R)\) directly. This proves
\eqref{content:exactPhi}.
\end{proof}

\begin{proof}[Proof of Corollary~\ref{content:basic}]
By Lemma~\ref{lem:capacity-packing}, every labelled cover
$A\subseteq\bigcup_j A_j$ satisfies
\[
 \mathcal P^s(A)\le\sum_j\mathcal P^s_0(A_j)
 \le\sum_j\mathcal P^s_{R_j}(A_j)
 =\sum_j\Phi^s(A_j,R_j).
\]
Taking the infimum gives the first inequality in~\eqref{content:chain};
the next two follow from the admissible covers, and the last from
Theorem~\ref{thm:comparison}.
\end{proof}

When two ambient spaces occur, we indicate them by subscripts. 

\begin{proposition}
\label{prop:Phi-basic}
Let \(s>0\), let \(X\) and \(Y\) be metric spaces, and let
\(E\subset X\) and \(0<R<\infty\) satisfy \(|E|\le R\).
\begin{enumerate}
\item[\textup{(i)}]
If \(E\subseteq F\subseteq X\) and \(|F|\le R\), then
\[
        \Phi^s(E,R)\leq\Phi^s(F,R).
\]
Moreover,
\[
        \Phi^s(E,R)=\Phi^s(\overline E,R).
\]

\item[\textup{(ii)}]
If \(E_1\subseteq E_2\subseteq\cdots\) and
\(E=\bigcup_{k=1}^\infty E_k\), then
\[
        \Phi^s(E_k,R)\uparrow\Phi^s(E,R).
\]

\item[\textup{(iii)}]
The map
\[
        (K,R)\longmapsto\Phi^s(K,R),
        \qquad K\in\mathcal K(X),\quad |K|\le R<\infty,\quad R>0,
\]
is jointly lower semicontinuous, with the Hausdorff metric on the first
coordinate and the usual metric on the second.

\item[\textup{(iv)}]
If \(f:X\to Y\) is \(L\)-Lipschitz, \(L>0\), then
\[
        \Phi_Y^s(f(E),LR)\leq L^s\Phi_X^s(E,R).
\]

\item[\textup{(v)}]
If \(f:X\to Y\) is a surjective similarity of ratio \(\lambda>0\), then
\[
        \Phi_Y^s(f(E),\lambda R)=\lambda^s\Phi_X^s(E,R).
\]
\end{enumerate}
\end{proposition}

\begin{proof}
By Lemma~\ref{lem:capacity-packing}, the supremum defining
\(\Phi^s(E,R)\) may be taken over configurations centred in \(E\).

Assertion (i) follows from inclusion of the admissible configurations
and the identity
\(\overline{\overline E}^{\widehat X}=\overline E^{\widehat X}\).
For (ii), every finite configuration centred in \(E\) has all its centres
in some \(E_k\). Taking suprema and using monotonicity gives the result;
the case \(E=\varnothing\) is immediate.

For (iii), let \(K_m\to K\) in the Hausdorff metric and \(R_m\to R\),
where \(|K_m|\le R_m<\infty\) and \(R_m>0\). Given
\(0\le a<\Phi^s(K,R)\), choose a finite admissible configuration
\(\{(x_i,r_i)\}_{i=1}^N\) in \(K\) and \(\theta\in(0,1)\) such that
\[
        \sum_{i=1}^N(2\theta r_i)^s>a.
\]
Choose \(x_i^{(m)}\in K_m\) with \(x_i^{(m)}\to x_i\). For all
sufficiently large \(m\),
\[
        d(x_i^{(m)},x_j^{(m)})\ge\theta(r_i+r_j)
        \quad(i\ne j),
        \qquad \theta r_i\le\theta R<R_m.
\]
Consequently,
\[
        \liminf_{m\to\infty}\Phi^s(K_m,R_m)
        \ge\sum_{i=1}^N(2\theta r_i)^s>a,
\]
which proves lower semicontinuity.

For (iv), \(|f(E)|\le LR\). By Lemma~\ref{lem:capacity-packing}, it suffices to
consider finite admissible configurations
\(\{(f(x_i),\rho_i)\}_i\), with \(x_i\in E\) and \(\rho_i\le LR\).
Since
\[
        d_X(x_i,x_j)
        \ge L^{-1}d_Y(f(x_i),f(x_j))
        \ge \rho_i/L+\rho_j/L,
        \qquad \rho_i/L\le R,
\]
the configuration \(\{(x_i,\rho_i/L)\}_i\) is admissible for
\(\Phi_X^s(E,R)\). Taking suprema proves (iv), including the empty
carrier. Applying (iv) to \(f\) and \(f^{-1}\) gives (v).
\end{proof}

\begin{proposition}
\label{prop:M-basic-properties}
Let \(s>0\), and let \(X\) and \(Y\) be metric spaces.
\begin{enumerate}
\item[\textup{(i)}]
\(\mathcal M^s\) is a metric outer measure on \(X\). In particular, all
Borel subsets of \(X\) are \(\mathcal M^s\)-measurable.

\item[\textup{(ii)}]
\(\mathcal M^s\) is Borel regular: for every \(A\subseteq X\), there is
a Borel set \(B\supseteq A\) such that
\[
        \mathcal M^s(B)=\mathcal M^s(A).
\]

\item[\textup{(iii)}]
If \(f:X\to Y\) is \(L\)-Lipschitz, $L>0$, then
\[
        \mathcal M_Y^s(f(A))
        \leq L^s\mathcal M_X^s(A)
        \qquad(A\subseteq X).
\]

\item[\textup{(iv)}]
If \(f:X\to Y\) is a surjective similarity of ratio \(\lambda>0\), then
\[
        \mathcal M_Y^s(f(A))=\lambda^s\mathcal M_X^s(A).
\]
\end{enumerate}
\end{proposition}

\begin{proof}
For each \(0<\delta<1\), the covering construction defines an outer
measure \(\mathcal M^s_\delta\). Thus, for any sequence \((A_j)\),
\[
        \mathcal M^s_\delta\Bigl(\bigcup_j A_j\Bigr)
        \le\sum_j\mathcal M^s_\delta(A_j)
        \le\sum_j\mathcal M^s(A_j).
\]
Letting \(\delta\downarrow0\) proves countable subadditivity of
\(\mathcal M^s\); monotonicity and the empty-set condition follow from
the definition. If \(\dist(A,B)>0\) and
\(0<\delta<\min\{1,\dist(A,B)\}\), every carrier with label at most
\(\delta\) meets at most one of \(A\) and \(B\). Splitting any cover
of \(A\cup B\) accordingly gives
\[
        \mathcal M^s_\delta(A\cup B)
        =\mathcal M^s_\delta(A)+\mathcal M^s_\delta(B).
\]
Passing to the limit proves (i).

For (ii), assume first that \(\mathcal M^s(A)<\infty\). For integers
\(m\ge2\) and \(k\ge1\), choose labelled covers
\[
        A\subseteq\bigcup_{j=1}^\infty E_j^{m,k},
        \qquad |E_j^{m,k}|\le R_j^{m,k}\le\frac1m,
        \qquad R_j^{m,k}>0,
\]
with
\[
        \sum_{j=1}^\infty\Phi^s(E_j^{m,k},R_j^{m,k})
        \le\mathcal M^s_{1/m}(A)+2^{-(m+k)}.
\]
Proposition~\ref{prop:Phi-basic}(i) allows each carrier to be replaced
by its closure without changing its label or cost. Hence the
\(F_\sigma\) sets
\[
        F_{m,k}:=\bigcup_{j=1}^\infty\overline{E_j^{m,k}}
\]
contain \(A\) and satisfy
\[
        \mathcal M^s_{1/m}(F_{m,k})
        \le\mathcal M^s_{1/m}(A)+2^{-(m+k)}.
\]
The Borel set \(B:=\bigcap_{m\ge2,\,k\ge1}F_{m,k}\) contains \(A\).
For each \(m\), monotonicity and the preceding estimate give
\[
        \mathcal M^s_{1/m}(A)
        \le\mathcal M^s_{1/m}(B)
        \le\mathcal M^s_{1/m}(A)+2^{-(m+k)}
        \qquad(k\ge1).
\]
Letting \(k\to\infty\) and then \(m\to\infty\) proves (ii). If
\(\mathcal M^s(A)=\infty\), take \(B=X\).

For (iii), let \(L>0\) and \(0<\delta<\min\{1,1/L\}\). Replacing
each pair \((E_j,R_j)\) in a labelled \(\delta\)-cover of \(A\) by
\((f(E_j),LR_j)\), Proposition~\ref{prop:Phi-basic}(iv) gives
\[
        \mathcal M^s_{Y,L\delta}(f(A))
        \le L^s\mathcal M^s_{X,\delta}(A).
\]
Let \(\delta\downarrow0\).  Applying (iii) to \(f\)
and \(f^{-1}\) proves (iv).
\end{proof}

\begin{proposition}
\label{content:outer-properties}
Let \(s>0\) and \(0<\delta<\infty\).
\begin{enumerate}
\item[\textup{(i)}]
\(\MM^s_\infty\) and \(\MM^s_\delta\) are outer measures on \(\R^d\).

\item[\textup{(ii)}]
For every \(A\subseteq\R^d\) and
\(Q\in\{\MM^s_\infty,\MM^s_\delta\}\), there is a Borel set
\(B\supseteq A\) such that
\[
        Q(B)=Q(A).
\]

\item[\textup{(iii)}]
If \(T:\R^d\to\R^d\) is a similarity of ratio \(b>0\), then
\begin{equation}\label{content:covariance}
        \MM^s_\infty(TA)=b^s\MM^s_\infty(A)
        \qquad(A\subseteq\R^d).
\end{equation}
\end{enumerate}
\end{proposition}

\begin{proof}
Fix \(Q\in\{\MM^s_\infty,\MM^s_\delta\}\). For (i), monotonicity
and \(Q(\varnothing)=0\) follow from the definitions; concatenating
covers gives countable subadditivity.

For (ii), take \(B=\R^d\) if \(Q(A)=\infty\). Otherwise, choose
admissible covers \((E_j^k,R_j^k)_j\) of \(A\) with costs at most
\(Q(A)+1/k\), and put
\[
        B:=\bigcap_{k=1}^\infty\bigcup_{j=1}^\infty\overline{E_j^k}.
\]
Then \(B\) is Borel, \(A\subseteq B\), and
Proposition~\ref{prop:Phi-basic}(i) gives
\[
        Q(A)\le Q(B)
        \le\sum_j\Phi^s(E_j^k,R_j^k)
        \le Q(A)+\frac1k.
\]
Let \(k\to\infty\).

For (iii), \((E_j,R_j)_j\mapsto(TE_j,bR_j)_j\) bijects the
covers admissible for \(\MM^s_\infty(A)\) and \(\MM^s_\infty(TA)\).
Taking infima over the former, Proposition~\ref{prop:Phi-basic}(v) gives
\[
        \MM^s_\infty(TA)
        =\inf\sum_j\Phi^s(TE_j,bR_j)
        =b^s\inf\sum_j\Phi^s(E_j,R_j)
        =b^s\MM^s_\infty(A).
\]
\end{proof}

\subsection{Equality on subsets}

\begin{lemma}\label{content:transfer}
Let $\nu$ be a Borel regular metric outer measure and let $Q$ be an
outer measure on $\R^d$ admitting Borel supersets of equal value,
with $Q\le\nu$. If $E$ is
Borel and $Q(E)=\nu(E)<\infty$, then $Q(A)=\nu(A)$ for every
$A\subset E$.
\end{lemma}

\begin{proof}
For Borel $B\subset E$,
\[
 \nu(E)=Q(E)\le Q(B)+Q(E\setminus B)
 \le Q(B)+\nu(E\setminus B).
\]
Since $\nu(E)<\infty$, subtraction gives $\nu(B)\le Q(B)$.
For arbitrary $A\subset E$, choose a Borel set $B$ with
$A\subset B\subset E$ and $Q(B)=Q(A)$. Then
\[
 \nu(A)\le\nu(B)=Q(B)=Q(A)\le\nu(A).
\]
\end{proof}

For Hausdorff content, Farkas--Fraser~\cite[Lemma~1.1]{FF} state
inheritance for Hausdorff-measurable subsets. Lemma~\ref{content:transfer}
also gives the arbitrary-subset version: Borel hulls for
$\mathcal H^s_\infty$ are obtained by closing the carriers of
near-optimal covers and intersecting the resulting countable unions
of closed sets.

\subsection{Measurability on compact hyperspaces}
\label{compact hyperspaces}

On compact hyperspaces, Hausdorff content is upper semicontinuous,
hence of Baire class~1, whereas Hausdorff measure is of Baire class~2
and need not be of Baire class~1~\cite[Theorem~2.1]{MM}.
This motivates comparing the measurability of labelled packing content
and labelled packing measure. We show that, on \(\K(\R^d)\),
labelled packing content has no better measurability properties than
labelled packing measure with respect to the Borel \(\sigma\)-algebra,
\(\sigma(\Sigma^1_1)\), and the universal completion: their 
maps are measurable simultaneously with respect to each of these
three \(\sigma\)-algebras.

Fix \(s>0\) and a Polish metric space \(X\). Equip \(\K(X)\)
with the Hausdorff metric. Adjoin
\(\varnothing\) as an isolated point to obtain \(\Kz(X)\), and
assign both maps the value zero there. Let \(\Prob(X)\)
denote the space of Borel probabilities with the weak topology.
All maps below take values in \([0,\infty]\) with its
order topology.

For \(K\in\K(X)\) and \(0<\delta\le\infty\),
\begin{equation}
\label{eq:compact-cover}
        \MM^s_\delta(K)
        =\inf\left\{\sum_j\Phi^s(L_j,R_j):
        \begin{array}{l}
          L_j\in\Kz(X),\quad L_j\subseteq K,\\
          |L_j|\le R_j\le\delta,\quad 0<R_j<\infty,\\
          K=\displaystyle\bigcup_jL_j
        \end{array}\right\}.
\end{equation}
Indeed, replacing \((E_j,R_j)\) by
\((\overline{E_j\cap K},R_j)\) preserves admissibility and does not
increase the cost, by Proposition~\ref{prop:Phi-basic}(i).
The case \(\delta=\infty\) gives \(\MM^s_\infty(K)\).

\begin{lemma}
\label{content:hyperspace-test}
For fixed \(0<\delta\le\infty\), let \(\mathscr E_\delta\)
consist of the tuples
\((K,a,\mu,t)\in\K(X)\times(0,\infty)\times\Prob(X)\times\R\)
such that \(\mu(K)=1\) and
\begin{equation}
\label{content:global-test}
        \sum_{i=1}^m\Phi^s(L_i,R_i)
        +a\left(1-\mu\Bigl(\bigcup_{i=1}^mL_i\Bigr)\right)
        \ge t
\end{equation}
for every finite family satisfying
\[
        L_i\in\Kz(X),\quad L_i\subseteq K,\quad
        |L_i|\le R_i\le\delta,\quad 0<R_i<\infty.
\]
Then \(\mathscr E_\delta\) is Borel.
\end{lemma}

\begin{proof}
The map \((\mu,L)\mapsto\mu(L)\) is upper semicontinuous, so
\(\mu(K)=1\) is closed and the left-hand side
of~\eqref{content:global-test} is lower semicontinuous by
Proposition~\ref{prop:Phi-basic}(iii). The requirement
that~\eqref{content:global-test} hold for every admissible family
fails precisely when there exist integers \(m,\ell,j\ge1\) and an
admissible family \((L_i,R_i)_{i=1}^m\) such that
\[
        \begin{gathered}
        \frac1\ell\le R_i\le\min\{\ell,\delta\}
        \qquad(1\le i\le m),\\
        \sum_{i=1}^m\Phi^s(L_i,R_i)
        +a\left(1-\mu\Bigl(\bigcup_{i=1}^mL_i\Bigr)\right)
        \le t-\frac1j.
        \end{gathered}
\]
For fixed \(m,\ell,j\), these conditions define a closed set of tuples.
Its projection onto \((K,a,\mu,t)\) is also closed: if base tuples
converge with \(K_n\to K\), their carriers lie in the compact set
\(K\cup\bigcup_nK_n\), so the witnesses have a convergent
subsequence. The complement of the universal condition is therefore
\(F_\sigma\), proving the claim.
\end{proof}

We first use this condition to establish measurability of labelled
packing measure on a general Polish space.

\begin{theorem}
\label{thm:main}
For every \(c\in\R\), the set
\[
        \{K\in\K(X):\MM^s(K)\ge c\}
\]
is analytic. Consequently, \(K\mapsto\MM^s(K)\) is
\(\sigma(\Sigma^1_1(\K(X)))\)-measurable and universally measurable.
\end{theorem}

\begin{proof}
It suffices to consider \(c>0\). If \(\MM^s(K)\ge c\), there is
a compact \(N\subseteq K\) with
\begin{equation}
\label{eq:finite-core}
        c\le\MM^s(N)<\infty.
\end{equation}
For finite \(\MM^s(K)\), take \(N=K\). Otherwise, the compact-subset
theorem of Joyce and Preiss~\cite{JP} gives \(N\subseteq K\) with
\(2^s c\le\mathcal P^s(N)<\infty\); Theorem~\ref{thm:comparison}
then gives~\eqref{eq:finite-core}.

For such \(N\) and every Borel \(A\subseteq N\), subadditivity gives
\[
        \MM^s_\delta(N)
        \le\MM^s_\delta(A)+\MM^s(N\setminus A),
\]
and hence
\begin{equation}
\label{eq:defect}
        0\le\MM^s(A)-\MM^s_\delta(A)
        \le\MM^s(N)-\MM^s_\delta(N)\longrightarrow0
        \qquad(\delta\downarrow0).
\end{equation}
Define the Borel set
\[
        \mathscr A_c:=
        \bigcap_{n\ge1}\ \bigcup_{m\ge\max\{n,2\}}
        \{(K,\mu):(K,c,\mu,c-2^{-n})\in\mathscr E_{1/m}\}.
\]
We claim that
\begin{equation}
\label{eq:projection}
        \{K\in\K(X):\MM^s(K)\ge c\}
        =\proj_{\K(X)}\mathscr A_c.
\end{equation}
For the forward inclusion, take \(N\) as above and
\(\mu=\MM^s|_N/\MM^s(N)\). Given \(n\), choose
\(m\ge\max\{n,2\}\) with
\(\MM^s(N)-\MM^s_{1/m}(N)<2^{-n}\). For any finite family in
the definition of \(\mathscr E_{1/m}\), put \(U=\bigcup_iL_i\).
By~\eqref{eq:defect},
\[
        c\mu(U)\le\MM^s(N\cap U)
        \le\MM^s_{1/m}(N\cap U)+2^{-n}
        \le\sum_i\Phi^s(L_i,R_i)+2^{-n}.
\]
Thus \((K,c,\mu,c-2^{-n})\in\mathscr E_{1/m}\).
Conversely, let \((K,\mu)\in\mathscr A_c\). For each \(n\ge1\),
choose \(m\ge\max\{n,2\}\) such that
\((K,c,\mu,c-2^{-n})\in\mathscr E_{1/m}\).
Apply inequality~\eqref{content:global-test}, with \(a=c\) and
\(t=c-2^{-n}\), to the first \(q\) members of any cover
in~\eqref{eq:compact-cover} with \(\delta=1/m\).
Since \(\mu(\bigcup_{i=1}^qL_i)\uparrow1\),
\[
        c-2^{-n}\le\sum_i\Phi^s(L_i,R_i).
\]
Taking infima gives \(\MM^s(K)\ge\MM^s_{1/m}(K)\ge c-2^{-n}\).
Let \(n\to\infty\). This proves~\eqref{eq:projection}, and the
measurability assertions follow since analytic sets are universally
measurable.
\end{proof}

For \(X=\R^d\), the same Borel condition allows us to recover
labelled packing content from labelled packing measure. A countable
reconstruction gives the converse.

\begin{proposition}
\label{content:global-measurability}
Let \(d\ge1\), \(X=\R^d\), and define
\[
        f(K):=\MM^s(K),\qquad g(K):=\MM^s_\infty(K)
        \qquad(K\in\K(X)).
\]
\begin{enumerate}
\item[\textup{(i)}]
\(f\) is Borel measurable if and only if \(g\) is Borel measurable.

\item[\textup{(ii)}]
The same equivalence holds for measurability with respect to
\(\sigma(\Sigma^1_1(\K(X)))\) and the universal completion
\(\mathcal U(\K(X))\). Both maps are measurable with respect to
these two \(\sigma\)-algebras.
\end{enumerate}
\end{proposition}

\begin{proof}
By~\eqref{content:chain}, there is \(C_s\ge1\) such that
\[
        C_s^{-1}\MM^s\le\MM^s_\infty\le\MM^s_\delta\le\MM^s
        \qquad(\delta>0).
\]
Extend \(f\) and \(g\) by zero to \(\Kz(\R^d)\). Let
\(\mathscr S\) be the Borel \(\sigma\)-algebra on \(\Kz(\R^d)\),
\(\sigma(\Sigma^1_1(\Kz(\R^d)))\), or
\(\mathcal U(\Kz(\R^d))\). Since \(\varnothing\) is isolated,
these extensions preserve the respective measurability properties.

Let \(\mathscr R\) be the countable family of bounded closed boxes
with rational endpoints, including degenerate boxes.
For \(B\in\mathscr R\), define
\(I_B:\Kz(\R^d)\to\Kz(\R^d)\) by \(I_B(K)=K\cap B\).
If \(U\subseteq\R^d\) is open, write \(B\cap U=\bigcup_{n\ge1}F_n\)
with each \(F_n\) compact. Then
\[
        \begin{aligned}
        \{K:I_B(K)\cap U\ne\varnothing\}
        &=\bigcup_{n\ge1}\{K:K\cap F_n\ne\varnothing\},\\
        \{K:I_B(K)\subseteq U\}
        &=\{K:K\cap(B\setminus U)=\varnothing\}.
        \end{aligned}
\]
The first set is \(F_\sigma\) and the second is open; hence \(I_B\)
is Borel measurable. Moreover,
\[
        E\in\mathscr S\quad\Longrightarrow\quad
        I_B^{-1}(E)\in\mathscr S.
\]
For \(\sigma(\Sigma^1_1)\), this follows because Borel inverse images
of analytic sets are analytic. For universal measurability, apply
the definition to the probability \((I_B)_*\lambda\), for each
Borel probability \(\lambda\) on \(\Kz(\R^d)\).

Suppose that \(f\) is \(\mathscr S\)-measurable, and let
\(H=\{K:0<f(K)<\infty\}\). For \(K\in H\), define
\[
        \mu_K:=\frac{\MM^s|_K}{f(K)}.
\]
For every \(B\in\mathscr R\),
\[
        \mu_K(B)=\frac{f(I_B(K))}{f(K)}.
\]
The maps \(\mu\mapsto\mu(B)\), \(B\in\mathscr R\), generate the
Borel \(\sigma\)-algebra of \(\Prob(\R^d)\): every open set is a
countable union of such boxes, and measures of finite unions are
determined by inclusion--exclusion. Thus \(K\mapsto\mu_K\) is
\(\mathscr S\)-measurable on \(H\).

For \(K\in H\) and \(c>0\),
\[
        g(K)\ge c\quad\Longleftrightarrow\quad
        (K,f(K),\mu_K,c)\in\mathscr E_\infty.
\]
If \(g(K)\ge c\), let \(L_i\in\Kz(\R^d)\), \(L_i\subseteq K\),
and \(|L_i|\le R_i<\infty\), \(R_i>0\), for \(1\le i\le m\).
With \(U=\bigcup_{i=1}^mL_i\), subadditivity gives
\[
        \begin{aligned}
        c&\le\MM^s_\infty(K)\\
        &\le\MM^s_\infty(U)+\MM^s_\infty(K\setminus U)\\
        &\le\sum_{i=1}^m\Phi^s(L_i,R_i)
        +f(K)(1-\mu_K(U)).
        \end{aligned}
\]
This proves \((K,f(K),\mu_K,c)\in\mathscr E_\infty\).
Conversely, suppose that \((K,f(K),\mu_K,c)\in\mathscr E_\infty\),
and let \((L_i,R_i)_{i\ge1}\) be any cover
in~\eqref{eq:compact-cover} with \(\delta=\infty\).
Inequality~\eqref{content:global-test}, with \(a=f(K)\) and \(t=c\),
gives
\[
        c\le\sum_{i=1}^q\Phi^s(L_i,R_i)
        +f(K)\left(1-\mu_K\Bigl(\bigcup_{i=1}^qL_i\Bigr)\right)
        \qquad(q\ge1).
\]
Since \(\mu_K(\bigcup_{i=1}^qL_i)\uparrow\mu_K(K)=1\), letting
\(q\to\infty\) and taking the infimum over these covers gives
\(c\le\MM^s_\infty(K)\).
By Lemma~\ref{content:hyperspace-test} and~\eqref{content:chain},
\[
        \begin{split}
        \{K:g(K)\ge c\}
        =\{K:f(K)=\infty\}
        \cup\{K\in H:(K,f(K),\mu_K,c)\in\mathscr E_\infty\}
        \in\mathscr S.
        \end{split}
\]
Thus \(g\) is \(\mathscr S\)-measurable.

To prove the reverse implication, first observe that, for
\(A\subseteq\R^d\) with \(|A|\le\delta\), the replacement
\[
        (E_j,R_j)\longmapsto
        (E_j\cap A,\min\{R_j,\delta\})
\]
transforms every cover admissible for \(\MM^s_\infty(A)\) into one
admissible for \(\MM^s_\delta(A)\) without increasing its cost.
Consequently,
\[
        \MM^s_\infty(A)=\MM^s_\delta(A)\qquad(|A|\le\delta).
\]
Let \(\mathfrak D\) be the countable set of finite families of
pairwise disjoint boxes in \(\mathscr R\). We shall prove
\begin{equation}
\label{content:box-recovery}
        \MM^s(K)=\sup_{\mathcal D\in\mathfrak D}
        \sum_{B\in\mathcal D}\MM^s_\infty(K\cap B).
\end{equation}
For every \(\mathcal D\in\mathfrak D\), finite additivity of
\(\MM^s\) on disjoint Borel sets gives
\[
        \sum_{B\in\mathcal D}\MM^s_\infty(K\cap B)
        \le\sum_{B\in\mathcal D}\MM^s(K\cap B)\le\MM^s(K).
\]
If \(\MM^s(K)=0\), both sides of~\eqref{content:box-recovery} vanish.
If \(\MM^s(K)=\infty\), choose \(B\in\mathscr R\) with
\(K\subseteq B\). Then \(\MM^s_\infty(K\cap B)=\infty\)
by~\eqref{content:chain}, proving~\eqref{content:box-recovery} in this case.

Suppose that \(0<\MM^s(K)<\infty\), and fix \(\varepsilon>0\).
Choose \(\delta>0\) such that
\(\MM^s_\delta(K)>\MM^s(K)-\varepsilon\). The finite measure
\(\MM^s|_K\) assigns positive measure to at most countably many
hyperplanes in each coordinate direction. Hence a cubical grid of
side length \(h<\delta/\sqrt d\) can be translated so that all its
boundary hyperplanes have measure zero. Let \(C_1,\ldots,C_N\)
be the open cubes of this grid that intersect \(K\). For each \(i\),
choose an increasing sequence of rational closed boxes with union
\(C_i\). Continuity from below yields \(B_i\in\mathscr R\),
\(B_i\subset C_i\), such that
\[
        \MM^s\bigl(K\cap(C_i\setminus B_i)\bigr)
        <\frac{\varepsilon}{N}\qquad(1\le i\le N).
\]
The family \(\{B_1,\ldots,B_N\}\) belongs to \(\mathfrak D\), and
\[
        |B_i|\le\delta,\qquad
        \MM^s\Bigl(K\setminus\bigcup_{i=1}^N B_i\Bigr)<\varepsilon.
\]
Subadditivity and the equality for sets of diameter at most \(\delta\)
therefore give
\begin{align}
\MM^s(K)-\varepsilon
        &<\MM^s_\delta(K)\nonumber\\
        &\le\sum_{i=1}^N\MM^s_\delta(K\cap B_i)
        +\MM^s_\delta\Bigl(K\setminus\bigcup_{i=1}^N B_i\Bigr)\nonumber\\
        &<\sum_{i=1}^N\MM^s_\infty(K\cap B_i)+\varepsilon.
        \nonumber
\end{align}
Letting \(\varepsilon\downarrow0\) proves~\eqref{content:box-recovery}.
If \(g\) is \(\mathscr S\)-measurable, each function
\(K\mapsto\sum_{B\in\mathcal D}g(I_B(K))\) is
\(\mathscr S\)-measurable. Equation~\eqref{content:box-recovery}
expresses \(f\) as the supremum of this countable family, so \(f\)
is \(\mathscr S\)-measurable.

Theorem~\ref{thm:main} gives
\(\sigma(\Sigma^1_1(\K(\R^d)))\)-measurability and universal
measurability of \(f\). The equivalence just proved gives both
properties for \(g\), completing (ii).
\end{proof}

\section{\texorpdfstring{Proof of Theorem~\ref{thm:comparison}}{Proof of the comparison theorem}}
\label{sec:main-result}

For a ball \(B=B(x,r)\) and \(a>0\), write \(aB:=B(x,ar)\).

\begin{lemma}
\label{lem:5r}
Let \(X\) be a metric space, and let \(\mathscr B\) be a family of balls
whose radii are positive and uniformly bounded above. There is a pairwise
disjoint subfamily \(\mathscr G\subseteq\mathscr B\) such that
\[
        \bigcup_{B\in\mathscr B}B
        \subseteq
        \bigcup_{B\in\mathscr G}5B.
\]
If the set of centres of the balls in \(\mathscr B\) is separable, then
\(\mathscr G\) may be chosen countable.
\end{lemma}

Finiteness of packing premeasure supplies the total boundedness needed
for a countable selection.

\begin{lemma}
\label{lem:finite-scale-finiteness}
Let \(s>0\), let \(X\) be a metric space, and let \(E\subset X\).

\begin{enumerate}
\item[\textup{(i)}]
If \(E\) is unbounded, then
\[
        \mathcal P^s_0(E)=\infty.
\]

\item[\textup{(ii)}]
If \(\mathcal P^s_0(E)<\infty\), then \(E\) is totally bounded and
\[
        \mathcal P^s_\delta(E)<\infty
        \qquad\text{for every }0<\delta<1.
\]
\end{enumerate}
\end{lemma}

\begin{proof}
For (i), fix \(0<\delta<1\). An unbounded set contains a sequence
\((x_k)\) with \(d(x_k,x_\ell)>2\delta\) for \(k\neq\ell\). The balls
\(B(x_k,\delta)\) form a \(\delta\)-packing, and hence
\(\mathcal P^s_\delta(E)=\infty\).

For (ii), choose \(\eta>0\) such that
\(\mathcal P^s_\eta(E)<\infty\). For every \(\varepsilon>0\), a family of \(N\) points in \(E\) with
pairwise distances at least \(\varepsilon\) satisfies
\[
        N\bigl(2\min\{\eta,\varepsilon/3\}\bigr)^s
        \leq\mathcal P^s_\eta(E).
\]
Thus every maximal \(\varepsilon\)-separated subset is finite, and
\(E\) is totally bounded.

Fix \(0<\delta<1\), and decrease \(\eta\), if necessary, so that
\(\eta\leq\delta\). Given a \(\delta\)-packing of \(E\), the balls of
radius at most \(\eta\) have total \(s\)-mass at most
\(\mathcal P^s_\eta(E)\). Shrinking every remaining ball to radius
\(\eta\) gives an \(\eta\)-packing, so the number \(N\) of remaining
balls satisfies
\[
        N(2\eta)^s\leq\mathcal P^s_\eta(E).
\]
Consequently,
\[
        \mathcal P^s_\delta(E)
        \leq
        \left(1+\left(\frac{\delta}{\eta}\right)^s\right)
        \mathcal P^s_\eta(E)<\infty.
\]
\end{proof}

\begin{remark}
Feng, Hua and Wen~\cite{FengHuaWen1999} proved that
\(\mathcal P^s(K)=\mathcal P^s_0(K)\) whenever
\(K\subset\mathbb R^n\) is compact, \(0\leq s\leq n\), and
\(\mathcal P^s_0(K)<\infty\). Their theorem concerns the relation between
the infinitesimal premeasure and its Method I regularization.
\end{remark}

\begin{lemma}
\label{lem:Ps-le-Ms}
Let \(s>0\), let \(X\) be a metric space, and let \(F\subset X\). Then
\[
        \mathcal P^s(F)\leq 2^s\mathcal M^s(F).
\]
\end{lemma}

\begin{proof}
Fix \(0<\delta<1\), and let \(\{(E_j,R_j)\}_j\)
be an admissible labelled cover of \(F\), so that
\(|E_j|\leq R_j\leq\delta\) and \(R_j>0\).
Lemma~\ref{lem:capacity-packing} gives
\[
        \mathcal P^s_0(E_j)
        \leq
        \mathcal P^s_{R_j}(E_j)
        \leq
        2^s\Phi^s(E_j,R_j).
\]
Taking the infimum over admissible labelled covers gives
\[
        \mathcal P^s(F)
        \leq2^s\mathcal M^s_\delta(F)
        \leq2^s\mathcal M^s(F).
\]
\end{proof}

For the reverse inequality, we decompose a set into a high-density
remainder and a part admitting a labelled cover of controlled cost.
For \(E\subset X\), write
\[
        D_E(x,r)
        :=
        \frac{\mathcal P^s_r(E\cap B(x,r))}{(2r)^s}.
\]

\begin{lemma}
\label{lem:local-density-decomposition}
Let \(s>0\), let \(X\) be a metric space, and suppose that
\(E\subset X\) satisfies \(\mathcal P^s_0(E)<\infty\). Fix
\(0<\delta<1/2\) and \(K>4^s\). Then there are a set \(G\subset E\), an at most countable index set
\(J\), and balls
\[
        B_j=B(x_j,r_j),
        \qquad j\in J,
\]
such that, writing
\[
        Z:=E\setminus G,
        \qquad
        E_j:=E\cap B_j,
\]
the following hold:

\begin{enumerate}
\item[\textup{(a)}]
\[
        G\subseteq\bigcup_{j\in J} E_j
        \quad\text{and}\quad
        \mathcal P^s_\delta(Z)
        \leq
        \frac{4^s}{K}\mathcal P^s_\delta(E).
\]

\item[\textup{(b)}]
For every \(j\), \(0<r_j\leq\delta\), and
\[
        \mathcal P^s_{r_j}(E\cap B_j)
        \leq
        K(2r_j)^s.
\]

\item[\textup{(c)}]
The balls \(B(x_j,r_j/5)\) are pairwise disjoint, and
\[
        \sum_{j\in J}(2r_j)^s
        \leq
        5^s\mathcal P^s_\delta(E).
\]

\item[\textup{(d)}]
For every \(j\), 
\[
        \Phi^s(E_j,2r_j)
        \leq
        K(1+2^s)(2r_j)^s.
\]
\end{enumerate}
Consequently,
\[
        \sum_{j\in J}\Phi^s(E_j,2r_j)
        \leq
        5^sK(1+2^s)\mathcal P^s_\delta(E).
\]
\end{lemma}

\begin{proof}
By Lemma~\ref{lem:finite-scale-finiteness}, \(E\) is totally bounded and
\(\mathcal P^s_\delta(E)<\infty\). Define
\[
        Z:=\{x\in E:D_E(x,r)>K\text{ for every }0<r\leq\delta\},
        \qquad G:=E\setminus Z.
\]
Let \(\{B(x_i,r_i)\}_{i\in I}\) be a \(\delta\)-packing of \(Z\).
For each \(i\in I\),
\[
        \mathcal P^s_{r_i/4}
        \left(E\cap B\left(x_i,\frac{r_i}{4}\right)\right)
        >K4^{-s}(2r_i)^s.
\]
Choose a packing
\(\mathscr P_i=\{B(x_{i,k},\rho_{i,k})\}_{k\in I_i}\) with
\(x_{i,k}\in E\cap B(x_i,r_i/4)\) and
\(0<\rho_{i,k}\leq r_i/4\), such that
\[
        \sum_{k\in I_i}(2\rho_{i,k})^s>K4^{-s}(2r_i)^s.
\]
Since \(B(x_{i,k},\rho_{i,k})\subseteq B(x_i,r_i/2)\), the union
\(\bigcup_{i\in I}\mathscr P_i\) is a \(\delta\)-packing of \(E\).
Hence
\[
        \mathcal P^s_\delta(E)
        \geq\sum_{i\in I}\sum_{k\in I_i}(2\rho_{i,k})^s
        \geq K4^{-s}\sum_{i\in I}(2r_i)^s.
\]
Taking the supremum over \(\delta\)-packings of \(Z\) proves the
precontent estimate in \textup{(a)}.

For each \(x\in G\), choose \(r(x)\in(0,\delta]\) such that
\[
        \mathcal P^s_{r(x)}\bigl(E\cap B(x,r(x))\bigr)
        \leq K(2r(x))^s.
\]
Apply Lemma~\ref{lem:5r} to
\[
        \mathscr B
        :=\left\{B\left(x,\frac{r(x)}5\right):x\in G\right\}.
\]
Its radii are at most \(\delta/5\), and its set of centres is separable
because \(G\subset E\) is totally bounded. We obtain an at most
countable, pairwise disjoint family
\(\{B(x_j,r_j/5)\}_{j\in J}\), with \(x_j\in G\) and
\(r_j=r(x_j)\), such that
\[
        G\subseteq\bigcup_{j\in J}\bigl(E\cap B(x_j,r_j)\bigr)
        =\bigcup_{j\in J}E_j.
\]
This proves the covering assertion in \textup{(a)}, and the choice of
\(r_j\) gives \textup{(b)}. The selected family is a
\(\delta\)-packing of \(E\), so
\[
        \sum_{j\in J}(2r_j)^s
        =5^s\sum_{j\in J}\left(2\frac{r_j}{5}\right)^s
        \leq5^s\mathcal P^s_\delta(E),
\]
which proves \textup{(c)}.

Fix \(j\in J\). Since \(E_j\subseteq B(x_j,r_j)\),
\(|E_j|\leq2r_j\leq2\delta<1\), so \((E_j,2r_j)\) is an
admissible labelled carrier. Let \(\{B(y_k,\rho_k)\}_{k\in I}\)
be a \(2r_j\)-packing of \(E_j\) by disjoint closed balls, and write
\[
        I_{\leq}:=\{k\in I:\rho_k\leq r_j\},
        \qquad I_{>}:=\{k\in I:\rho_k>r_j\}.
\]
By \textup{(b)},
\[
        \sum_{k\in I_{\leq}}(2\rho_k)^s
        \leq\mathcal P^s_{r_j}(E_j)
        \leq K(2r_j)^s.
\]
Shrinking the balls indexed by \(I_{>}\) to radius \(r_j\) gives an
\(r_j\)-packing of \(E_j\). Consequently,
\[
        N:=\#I_{>}\leq
        \frac{\mathcal P^s_{r_j}(E_j)}{(2r_j)^s}
        \leq K<\infty.
\]
Since \(\rho_k\leq2r_j\),
\[
\begin{aligned}
        \sum_{k\in I}(2\rho_k)^s
        \leq K(2r_j)^s+N(4r_j)^s
        \leq K(1+2^s)(2r_j)^s.
\end{aligned}
\]
Taking the supremum and applying Lemma~\ref{lem:capacity-packing} gives
\[
        \Phi^s(E_j,2r_j)
        \leq\mathcal P^s_{2r_j}(E_j)
        \leq K(1+2^s)(2r_j)^s.
\]
This includes singleton carriers and proves \textup{(d)}. The final
estimate follows by summing \textup{(d)} and applying \textup{(c)}.
\end{proof}

Iterating the decomposition at a fixed scale gives a cover of the
whole set.

\begin{lemma}
\label{lem:local-cost}
For every \(s>0\), there is a constant
\(C_{\mathrm{loc}}(s)<\infty\) such that, for every metric space \(X\),
every \(E\subset X\), and every \(0<\delta<1\),
\[
        \mathcal M^s_{\delta}(E)
        \leq
        C_{\mathrm{loc}}(s)\mathcal P^s_{\delta}(E).
\]
\end{lemma}

\begin{proof}
If \(\mathcal P^s_0(E)=\infty\), the
assertion is immediate. Assume \(\mathcal P^s_0(E)<\infty\), put
\(\delta_0:=\delta/2\), fix \(K>4^s\), and write
\[
        \lambda:=\frac{4^s}{K},
        \qquad
        A_s(K):=5^sK(1+2^s).
\]
Starting with \(E^{(0)}:=E\), apply
Lemma~\ref{lem:local-density-decomposition} successively to
\(E^{(m)}\) at scale \(\delta_0\). Thus
\[
        E^{(m)}=G^{(m)}\cup E^{(m+1)},
\]
where \(G^{(m)}\) has a labelled cover
\(\{(E_j^{(m)},2r_j^{(m)})\}_j\) satisfying
\[
        |E_j^{(m)}|\leq2r_j^{(m)}\leq\delta,
        \qquad r_j^{(m)}>0,
\]
\[
        \sum_j\Phi^s(E_j^{(m)},2r_j^{(m)})
        \leq
        A_s(K)\mathcal P^s_{\delta_0}(E^{(m)}),
\]
and
\[
        \mathcal P^s_{\delta_0}(E^{(m+1)})
        \leq
        \lambda\mathcal P^s_{\delta_0}(E^{(m)}).
\]

The part of \(E\) omitted by all \(E_j^{(m)}\) is contained in
\(\bigcap_mE^{(m)}\), and
\[
        \mathcal P^s_{\delta_0}\left(\bigcap_mE^{(m)}\right)
        \leq
        \lambda^m\mathcal P^s_{\delta_0}(E)
        \longrightarrow0.
\]
Every non-empty set has \(\delta_0\)-packing precontent at least
\((2\delta_0)^s\), so the intersection is empty. Consequently, the
pairs \((E_j^{(m)},2r_j^{(m)})\) form an admissible labelled cover
at scale \(\delta\), and \eqref{eq:label-content} gives
\begin{align*}
        \mathcal M^s_\delta(E)
        &\leq
        \sum_{m=0}^\infty\sum_j\Phi^s(E_j^{(m)},2r_j^{(m)})\\
        &\leq
        A_s(K)\sum_{m=0}^\infty
        \mathcal P^s_{\delta_0}(E^{(m)})\\
        &\leq
        \frac{A_s(K)}{1-\lambda}\mathcal P^s_{\delta_0}(E)\\
        &\leq
        \frac{A_s(K)}{1-\lambda}\mathcal P^s_\delta(E).
\end{align*}
For example, taking \(K=2\cdot4^s\) gives
\[
        C_{\mathrm{loc}}(s)
        =
        4\cdot4^s5^s(1+2^s).
\]
\end{proof}

Applying this estimate to a Method I cover gives the comparison of
outer measures.

\begin{lemma}
\label{lem:global-Ms-le-Ps}
For every \(s>0\), every metric space \(X\), and every \(F\subset X\),
\[
        \mathcal M^s(F)
        \leq
        C_{\mathrm{loc}}(s)\mathcal P^s(F).
\]
\end{lemma}

\begin{proof}
If \(\mathcal P^s(F)=\infty\), there is
nothing to prove. For every \(E\subset X\),
Lemma~\ref{lem:local-cost} gives
\[
        \mathcal M^s(E)
        \leq
        C_{\mathrm{loc}}(s)\mathcal P^s_0(E)
\]
by letting \(\delta\downarrow0\). Fix \(\varepsilon>0\), and choose a
countable cover \(F\subseteq\bigcup_iF_i\) such that
\[
        \sum_i\mathcal P^s_0(F_i)
        \leq
        \mathcal P^s(F)+\varepsilon.
\]
By countable subadditivity of \(\mathcal M^s\),
\[
\begin{aligned}
        \mathcal M^s(F)
        &\leq
        \sum_i\mathcal M^s(F_i)\\
        &\leq
        C_{\mathrm{loc}}(s)\sum_i\mathcal P^s_0(F_i)\\
        &\leq
        C_{\mathrm{loc}}(s)
        \bigl(\mathcal P^s(F)+\varepsilon\bigr).
\end{aligned}
\]
Letting \(\varepsilon\downarrow0\) proves the claim.
\end{proof}

\begin{proof}[Proof of Theorem~\ref{thm:comparison}]
The two inequalities follow from Lemmas~\ref{lem:Ps-le-Ms}
and~\ref{lem:global-Ms-le-Ps}. Thus \(\mathcal M^s\) and
\(\mathcal P^s\) have the same null sets for every \(s>0\), and their
critical exponents agree.
\end{proof}

\begin{remark}
\label{rem:s=0}
With the convention \((2r)^0=1\),
\[
        \Phi^0(E,R)=\#E
        \qquad (|E|\leq R,\ 0<R<1),
\]
where \(\#E=\infty\) for every infinite set. Indeed, any finite set
of distinct centres admits sufficiently small positive separated radii.
It follows directly from the definition that
\(\mathcal M^0=\mathcal P^0\) is counting measure.
\end{remark}

\section{\texorpdfstring{Proof of Theorem~\ref{content:subshift}}{Proof of the subshift theorem}}
\label{content:section}

Let $I$ be a finite alphabet and let $A=(A_{ij})_{i,j\in I}$ be a
zero--one matrix. Its one-sided subshift of finite type is
\[
 \Sigma_A=\{\alpha\in I^{\N}:A_{\alpha_k\alpha_{k+1}}=1 
 \text{ for every }k\ge1\}.
\]
The matrix is \emph{irreducible} if, for every $i,j\in I$, some
$n\ge1$ satisfies $(A^n)_{ij}>0$. Attach to each symbol $i$ a
contracting similarity $S_i$ of $\R^d$ with ratio $a_i\in(0,1)$.
The coding map is
\[
 \Pi(\alpha)=\lim_{n\to\infty}
 S_{\alpha_1}\circ\cdots\circ S_{\alpha_n}(x_0),
 \qquad \alpha\in I^{\N},
\]
where the limit is independent of $x_0\in\R^d$.
For a finite nonempty word $w=w_1\cdots w_n$, put
\[
 [w]=\{\alpha\in I^{\N}:\alpha_k=w_k\text{ for }1\le k\le n\},
\]
and write
\[
 F_A=\Pi(\Sigma_A),\qquad
 F_A^w=\Pi(\Sigma_A\cap[w]),\qquad
 G_i=\bigcup_{j:A_{ij}=1}F_A^j.
\]
Inadmissible words give empty cylinders. Then
\begin{equation}\label{content:follower-system}
 F_A^i=S_iG_i,\qquad G_i=\bigcup_{j:A_{ij}=1}S_jG_j.
\end{equation}

\subsection{A sufficient condition for equality}

We shall use the following condition on a non-singleton compact set
$E\subset\R^d$: there exist $\kappa>0$ and $0<r_0<\diam E$ such
that, for every $x\in E$ and $0<r\le r_0$, there is a similarity $T$
satisfying
\begin{equation}\label{content:local-copies}
 T(E)\subset E\cap\overline {B(x,r)},\qquad
 \kappa r\le\diam T(E)\le r.
\end{equation}
We do not require $x\in T(E)$ or disjointness of the sets $T(E)$.

\begin{lemma}\label{content:critical-finiteness}
If $E$ satisfies~\eqref{content:local-copies}, then
\[
 0<s:=\dimH E=\dimP E,\qquad
 \mathcal P^s(E)<\infty,\qquad \MM^s(E)<\infty.
\]
\end{lemma}

\begin{proof}
Put $D=\diam E$. Choose distinct $x_1,x_2\in E$ and
$0<r<\min\{r_0,|x_1-x_2|/3\}$. By~\eqref{content:local-copies},
there are similarities $T_1,T_2$, with ratios
$0<b_i\le r/D<1$, such that
\[
 T_i(E)\subset E\cap\overline {B(x_i,r)}\quad(i=1,2),\qquad
 \dist(T_1(E),T_2(E))\ge |x_1-x_2|-2r>0.
\]
Since $E$ is compact and $T_i(E)\subset E$, the self-similar
attractor $K$ of $\{T_1,T_2\}$ is contained in $E$.
Moreover, $T_1(K)\cap T_2(K)=\varnothing$, so the strong separation
condition holds. Consequently, $\dimH K$ is the unique solution
$u>0$ of $b_1^u+b_2^u=1$. Thus
$s=\dimH E\ge\dimH K=u>0$.

For any finite $r_0$-packing $(x_\ell,r_\ell)_{\ell=1}^N$ of $E$,
apply~\eqref{content:local-copies} at $x_\ell$ with radius $r_\ell/2$.
The resulting copies $T_\ell E$ are pairwise positively separated and
their ratios $b_\ell$ satisfy
\[
 T_\ell E\subset E\cap\overline B(x_\ell,r_\ell/2),\qquad
 b_\ell\ge\frac{\kappa r_\ell}{2D}.
\]
If $N\ge2$, the attractor $K$ of $(T_\ell)_{\ell=1}^N$ lies in $E$
and satisfies strong separation. Its dimension $u$ is determined by
$\sum_\ell b_\ell^u=1$~\cite{Falconer2014}. Since $u\le s$,
$\sum_\ell b_\ell^s\le1$; this inequality also holds for $N=1$.
Consequently,
\[
 \sum_{\ell=1}^N(2r_\ell)^s
 \le\left(\frac{4D}{\kappa}\right)^s\sum_{\ell=1}^N b_\ell^s
 \le\left(\frac{4D}{\kappa}\right)^s.
\]
Finite subsums give the same bound for countable packings. Hence
$\mathcal P^s_{r_0}(E)<\infty$, so $\mathcal P^s(E)<\infty$ and
$\dimP E\le s$. The reverse inequality follows from
$\dimH E\le\dimP E$, and $\MM^s(E)<\infty$ follows from
Theorem~\ref{thm:comparison}.
\end{proof}

\begin{lemma}\label{content:follower-copies}
If $A$ is irreducible, then either $G_i$ is a singleton for every
$i\in I$, or every $G_i$ satisfies~\eqref{content:local-copies}.
In the latter case,
\[
 0<\dimH G_i=\dimP G_i=\dimH F_A=\dimP F_A
 \qquad(i\in I).
\]
\end{lemma}

\begin{proof}
Associate to~\eqref{content:follower-system} the directed graph with
an edge $i\to j$ labelled by $S_j$ whenever $A_{ij}=1$. For a directed
path $p=(i_0,\ldots,i_n)$, define
\[
 S_p=S_{i_1}\circ\cdots\circ S_{i_n},\qquad
 a_p=\prod_{k=1}^n a_{i_k}.
\]
For the empty path, $S_p$ is the identity and $a_p=1$.
Then $S_p(G_{i_n})\subset G_{i_0}$.
Irreducibility implies that all $G_i$ have the same Hausdorff
dimension, and that either all are singletons or none is.
Assume the latter case and put $s=\dimH G_i$. Since
$F_A=\bigcup_i S_iG_i$, we also have $\dimH F_A=s$.

Fix $i$ and, for each $j$, choose a path from $j$ to $i$, with
similarity $U_j$ and ratio $c_j>0$; the empty path is allowed for
$j=i$. Put
\[
 D=\max_j\diam G_j>0,\qquad
 a=\min_j a_j>0,\qquad c=\min_j c_j>0.
\]
For $x\in G_i$ and $0<r\le(\diam G_i)/2$, choose an infinite
directed path from $i$ coding $x$, and let $p:i\to j$ be its
shortest prefix satisfying $a_pD\le r$. Then
\[
 a_p>ar/D,\qquad
 S_pU_jG_i\subset S_pG_j\subset G_i\cap\overline B(x,r),
\]
and
\[
 \frac{ac\,\diam G_i}{D}\,r
 <\diam(S_pU_jG_i)\le r.
\]
Thus $G_i$ satisfies~\eqref{content:local-copies}.
Lemma~\ref{content:critical-finiteness} gives $s>0$ and
$\dimP G_i=s$ for every $i$. Taking the finite union
$F_A=\bigcup_iS_iG_i$ gives $\dimP F_A=s$.
\end{proof}

\begin{lemma}\label{content:packing-density}
Let $s>0$, let $E\subset\R^d$ be Borel, and let $\mu$ be a finite
Borel measure supported on $E$. Suppose
$\mu(B)\le C\mathcal P^s(B)$ for every Borel $B\subset E$, where
$0\le C<\infty$. Then
\[
 \liminf_{r\downarrow0}\frac{\mu(B(x,r))}{(2r)^s}\le C
 \qquad\text{for $\mu$-almost every }x\in E.
\]
\end{lemma}

\begin{proof}
Fix $b>C$ and let
\[
 D_b=\left\{x\in E:
 \liminf_{r\downarrow0}\frac{\mu(B(x,r))}{(2r)^s}>b\right\}.
\]
The set $D_b$ is Borel and $\mu$ is Radon. Mattila's packing-density
estimate~\cite[Theorem~6.11]{Mattila1995}, with lower densities
normalised by $(2r)^s$, gives
\[
 b\mathcal P^s(D_b)\le\mu(D_b)\le C\mathcal P^s(D_b).
\]
Since $\mu$ is finite, $\mathcal P^s(D_b)<\infty$; hence $b>C$
forces $\mathcal P^s(D_b)=\mu(D_b)=0$. Let $b\downarrow C$
through a countable sequence.
\end{proof}

\begin{lemma}\label{content:return-exhaustion}
Suppose $E$ satisfies~\eqref{content:local-copies} and
$0<\MM^s(E)<\infty$, where $s>0$. For every $\theta>0$ there are
finitely or countably many similarities $T_n$, of ratios $b_n\le\theta$,
such that $T_nE\subset E$, $T_nE\cap T_mE=\varnothing\ (n\ne m)$, $\sum_n b_n^s=1$ and 
\begin{equation}\label{content:exhaustion}
 \begin{gathered}
 \MM^s\Bigl(E\setminus\bigcup_nT_nE\Bigr)=0.
 \end{gathered}
\end{equation}
\end{lemma}

\begin{proof}
Put $\mu=\MM^s|_E$, $m=\mu(E)$, $D=\diam E$,
$H=2/\kappa$, and $J=H+2$.
Fix $L>C(s)$. By Lemma~\ref{content:packing-density}, for
$\mu$-almost every $x\in E$ there are arbitrarily small $R>0$ such that
\begin{equation}\label{content:density-radii}
 \mu(B(x,R))\le L(2R)^s.
\end{equation}
For each such $x,R$ with $R/J\le\min\{r_0,\theta D\}$, put
$r=R/J$ and choose $T$ satisfying~\eqref{content:local-copies}.
Set $C=T(E)$ and $d_C=\diam C$. Then
\[
 \kappa r\le d_C\le r,\qquad
 \mu(C)=m(d_C/D)^s,
\]
and, writing $C^{[u]}=\{y:\dist(y,C)\le u\}$,
\begin{equation}\label{content:regular-copy}
 C^{[Hd_C]}\subset B(x,R),\qquad
 \mu(C^{[Hd_C]})\le
 M_0\mu(C),
\end{equation}
where $M_0=L(2JD/\kappa)^s/m$. Let $\mathscr C$ consist of the
sets $C$ obtained for all such $x,R$.

Choose pairwise disjoint $C_n\in\mathscr C$ recursively, with
\[
 \diam C_n>\frac12\sup\left\{\diam C:C\in\mathscr C,\quad
 C\cap\bigcup_{k<n}C_k=\varnothing\right\},
\]
until no member of $\mathscr C$ is disjoint from the preceding sets.
If the sequence is infinite, its diameters tend to zero,
since $\mu(C_n)=m(\diam C_n/D)^s$ and $\sum_n\mu(C_n)\le m$.
Thus every member $C$ of $\mathscr C$ intersects a selected $C_n$
with $\diam C\le2\diam C_n$.

Let $x\in E\setminus\bigcup_n C_n$ satisfy
\eqref{content:density-radii} for arbitrarily small $R>0$.
For each $N$ for which $C_1,\ldots,C_N$ are defined, choose such an $R$ with
$r=R/J<\dist(x,\bigcup_{n\le N}C_n)$ and
$r\le\min\{r_0,\theta D\}$. The associated set
$C\subset\overline {B(x,r)}$ meets some $C_n$, $n>N$,
with $\diam C\le2\diam C_n$. Hence
\[
 \dist(x,C_n)\le r\le\frac{\diam C}{\kappa}
 \le H\diam C_n.
\]
If the sequence is infinite, \eqref{content:regular-copy} gives
\[
 \mu\Bigl(E\setminus\bigcup_n C_n\Bigr)
 \le M_0\sum_{n>N}\mu(C_n)\longrightarrow0.
\]
If only $N$ sets are selected, this choice of $R$ gives a member of
$\mathscr C$ disjoint from all of them, contradicting the termination
of the construction. Finally, $C_n=T_nE$ and similarity covariance give
$m=\sum_n\mu(C_n)=m\sum_n b_n^s$.
\end{proof}

\begin{theorem}\label{content:local-copy-theorem}
Let $E\subset\R^d$ be non-singleton, compact, and satisfy
\eqref{content:local-copies}. Then $s:=\dimH E=\dimP E>0$, and
for every $B\subset E$ and every $0<\delta<\infty$,
\begin{equation}\label{content:general-equality}
 \MM^s_\infty(B)=\MM^s_\delta(B)=\MM^s(B)<\infty.
\end{equation}
\end{theorem}

\begin{proof}
Lemma~\ref{content:critical-finiteness} gives
$s=\dimH E=\dimP E>0$ and $\MM^s(E)<\infty$.
If $\MM^s(E)=0$, \eqref{content:chain} gives the assertion.
Suppose $\MM^s(E)>0$. 
Put $D=\diam E$ and fix $\eta,\varepsilon>0$.
By~\eqref{content:full}, choose a labelled cover
$(\widetilde A_j,\widetilde R_j)_j$ of $E$ with total cost at most
$\MM^s_\infty(E)+\varepsilon$. For each $j$, set
\[
 A_j=\widetilde A_j\cap E,\qquad
 R_j=\min\{\widetilde R_j,D\}.
\]
Then $E=\bigcup_j A_j$, $0<R_j\le D$, and $\diam A_j\le R_j$.
Since $A_j\subset\widetilde A_j$ and $R_j\le\widetilde R_j$,
every packing admissible for $\Phi^s(A_j,R_j)$ is also admissible
for $\Phi^s(\widetilde A_j,\widetilde R_j)$. Hence
\[
 \sum_j\Phi^s(A_j,R_j)
 \le\sum_j\Phi^s(\widetilde A_j,\widetilde R_j)
 \le\MM^s_\infty(E)+\varepsilon.
\]
Apply Lemma~\ref{content:return-exhaustion} with
$\theta=\eta/D$. The labelled family $(T_nA_j,b_nR_j)_{n,j}$
is an $\eta$-cover of $\bigcup_nT_nE$, since $b_nR_j\le\eta$.
By Proposition~\ref{prop:Phi-basic}\textup{(v)} and
$\sum_n b_n^s=1$, 
\begin{align*}
 \sum_{n,j}\Phi^s(T_nA_j,b_nR_j)
 &=\sum_n b_n^s\sum_j\Phi^s(A_j,R_j)\\
 &=\sum_j\Phi^s(A_j,R_j)\\
 &\le\MM^s_\infty(E)+\varepsilon.
\end{align*}
Moreover, \eqref{content:exhaustion} gives
$\MM^s(E\setminus\bigcup_nT_nE)=0$. Subadditivity and
\eqref{content:chain} therefore yield
\[
 \MM^s_\eta(E)
 \le\MM^s_\eta\Bigl(\bigcup_nT_nE\Bigr)
   +\MM^s\Bigl(E\setminus\bigcup_nT_nE\Bigr)
 \le\MM^s_\infty(E)+\varepsilon.
\]
Letting $\varepsilon\downarrow0$ and then $\eta\downarrow0$ gives
$\MM^s(E)=\MM^s_\infty(E)$. Lemma~\ref{content:transfer} and
\eqref{content:chain} give~\eqref{content:general-equality}
for every $B\subset E$.
\end{proof}

\subsection{Proof and consequences}

\begin{proof}[Proof of Theorem~\ref{content:subshift}]
Since $s=\dimH F_A>0$ and $F_A=\bigcup_iS_iG_i$, the sets $G_i$
cannot all be singletons. Lemma~\ref{content:follower-copies} gives
$\dimP F_A=s$ and~\eqref{content:local-copies} for every $G_i$.
Theorem~\ref{content:local-copy-theorem} therefore yields
\eqref{content:subshift-equality} for every $G_i$.
For an admissible word $w=w_1\cdots w_n$,
\[
 F_A^w=S_{w_1}\circ\cdots\circ S_{w_n}(G_{w_n}),
\]
so similarity covariance gives the cylinder assertion for every
finite $\delta>0$. Empty cylinders have zero content and measure.
\end{proof}

\begin{proof}[Proof of Corollary~\ref{content:strongM}]
For the full shift, $F_A=G_i=E$ for every $i\in I$.
Lemma~\ref{content:follower-copies} gives $s>0$, so
Theorem~\ref{content:subshift} applies.
\end{proof}

\begin{proof}[Proof of Corollary~\ref{content:graph}]
Let $I$ be the edge set and set $A_{ef}=1$ precisely when the terminal
vertex of $e$ is the initial vertex of $f$, as in
Farkas--Fraser~\cite[Proposition~2.5]{FF}. Strong connectivity implies
that $A$ is irreducible. If $e$ has terminal vertex $i$, then
\[
 F_A^e=S_e(E_i),\qquad
 G_e=\bigcup_{f:\,\text{initial vertex of }f=i}F_A^f=E_i.
\]
Every vertex has an incoming edge. Thus
Lemma~\ref{content:follower-copies} and Theorem~\ref{content:subshift}
apply to every $E_i$.
\end{proof}

\begin{corollary}\label{content:subsets}
Let $K$ be one of the sets covered by Theorem~\ref{content:subshift}
or Corollaries~\ref{content:strongM} and~\ref{content:graph}, with
corresponding critical exponent $s>0$. For every $B\subset K$ and
every finite $\delta>0$,
\begin{equation}\label{content:subset-equality}
 \MM^s_\infty(B)=\MM^s_\delta(B)=\MM^s(B)<\infty.
\end{equation}
\end{corollary}

\begin{proof}
Apply Lemma~\ref{content:transfer} with $Q=\MM^s_\infty$ and
$\nu=\MM^s$, using the Borel hulls from
Proposition~\ref{content:outer-properties}.
\end{proof}

\begin{remark}\label{content:zero}
With the convention \((2r)^0=1\), Theorem~\ref{content:subshift}
also holds for \(s=0\). Indeed, if \(A\) is irreducible and
\(\dimH F_A=0\), Lemma~\ref{content:follower-copies} implies that
every \(G_i\) is a singleton. Thus \(F_A=\bigcup_iS_iG_i\) is finite
and \(\dimP F_A=0\). Every \(F_A^w\) is a singleton or empty.
Thus, for \(K=F_A^w\) or \(K=G_i\) and every \(0<\delta<\infty\),
\[
        \MM^0_\infty(K)=\MM^0_\delta(K)=\MM^0(K)=\#K\le1.
\]
For a singleton \(K\), every admissible cover has cost at least one,
whereas the cover consisting of \((K,R)\), with \(0<R\le\delta\),
has cost one. For \(K=\varnothing\), all three values are zero.
\end{remark}

\section*{Acknowledgements}
The author thanks Lars Olsen for the invitation to visit the Analysis
Group at the University of St Andrews, where part of this work was
carried out, and Kenneth Falconer and Lars Olsen for helpful discussions. The visit
was supported by the China Scholarship Council
(award No.\ 202506840058).

\label{lastpage}


\begin{thebibliography}{99}

\bibitem{Falconer2014}
\textsc{K. J. Falconer}.
\emph{Fractal Geometry: Mathematical Foundations and Applications}, 3rd edn.
John Wiley \& Sons, Chichester, 2014.

\bibitem{FF}
\'A. Farkas and J. M. Fraser.
On the equality of Hausdorff measure and Hausdorff content.
\emph{J. Fractal Geom.} \textbf{2}(4) (2015), 403--429.
\href{https://doi.org/10.4171/JFG/27}{doi:10.4171/JFG/27}.

\bibitem{FengHuaWen1999}
\textsc{D.-J. Feng, S. Hua and Z.-Y. Wen}.
Some relations between packing premeasure and packing measure.
\emph{Bull. Lond. Math. Soc.} \textbf{31} (1999), 665--670.

\bibitem{JP}
\textsc{H. Joyce and D. Preiss}.
On the existence of subsets of finite positive packing measure.
\emph{Mathematika} \textbf{42} (1995), 15--24.

\bibitem{Mattila1995}
\textsc{P. Mattila}.
\emph{Geometry of Sets and Measures in Euclidean Spaces: Fractals and
Rectifiability}. Cambridge Studies in Advanced Mathematics vol.~44
(Cambridge University Press, 1995).

\bibitem{MM}
\textsc{P. Mattila and R. D. Mauldin}.
Measure and dimension functions: measurability and densities.
\emph{Math. Proc. Cambridge Philos. Soc.} \textbf{121} (1997), 81--100.

\bibitem{Rogers1970}
\textsc{C. A. Rogers}.
\emph{Hausdorff Measures} (Cambridge University Press, 1970).

\bibitem{SaintRaymondTricot1988}
\textsc{X. Saint Raymond and C. Tricot}.
Packing regularity of sets in \(n\)-space.
\emph{Math. Proc. Cambridge Philos. Soc.} \textbf{103} (1988), 133--145.

\bibitem{Sullivan1984}
\textsc{D. Sullivan}.
Entropy, Hausdorff measures old and new, and limit sets of geometrically finite Kleinian groups.
\emph{Acta Math.} \textbf{153} (1984), 259--277.

\bibitem{TaylorTricot1985}
\textsc{S. J. Taylor and C. Tricot}.
Packing measure and its evaluation for a Brownian path.
\emph{Trans. Amer. Math. Soc.} \textbf{288} (1985), 679--699.

\bibitem{Tricot1982}
\textsc{C. Tricot}.
Two definitions of fractional dimension.
\emph{Math. Proc. Cambridge Philos. Soc.} \textbf{91} (1982), 57--74.

\end{thebibliography}
\end{document}